\documentclass[a4paper,11pt]{amsart}

\usepackage{amsmath, amsthm, amsfonts, amssymb}
\usepackage[bookmarks=false]{hyperref}
\usepackage{enumerate}

\usepackage{color}

\numberwithin{equation}{section}

\newtheorem*{mainthm}{Main theorem}

\newtheorem{thm}{Theorem}[section]

\theoremstyle{definition}

\newtheorem{df}[thm]{Definition}

\newtheorem{claim}[thm]{Claim}

\newcommand{\R}{\mathbf{R}}
\newcommand{\C}{\mathbf{C}}

\newcommand{\N}{\mathbf{N}}

\newcommand{\cZ}{\mathcal{Z}}

\newcommand{\cX}{\mathcal{X}}
\newcommand{\cY}{\mathcal{Y}}
\newcommand{\cH}{\mathcal{H}}

\newcommand{\Ad}{\operatorname{Ad}}

\newcommand{\rL}{\mathord{\text{\rm L}}}
\newcommand{\rC}{\mathord{\text{\rm C}}}
\newcommand{\dom}{\mathord{\text{\rm dom}}}
\newcommand{\wm}{\mathord{\text{\rm wm}}}
\newcommand{\an}{\mathord{\text{\rm an}}}
\newcommand{\ap}{\mathord{\text{\rm ap}}}

\newcommand{\rE}{\mathord{\text{\rm E}}}

\newcommand{\Ball}{\mathord{\text{\rm Ball}}}

\newcommand{\dpr}{^{\prime\prime}}

\begin{document}

\title[Structure of modular invariant subalgebras in free Araki--Woods factors]{Structure of modular invariant subalgebras \\ in free Araki--Woods factors}

\begin{abstract}
We show that any amenable von Neumann subalgebra of any free Araki--Woods factor that is globally invariant under the modular automorphism group of the free quasi-free state is necessarily contained in the almost periodic free summand. 
\end{abstract}

\author{R\'emi Boutonnet}
\address{Institut de Math\'ematiques de Bordeaux \\ CNRS \\ Universit\'e Bordeaux I \\ 33405 Talence \\ FRANCE}
\email{remi.boutonnet@math.u-bordeaux1.fr}
\thanks{RB is supported by NSF Career Grant DMS 1253402}

\author{Cyril Houdayer}
\address{Laboratoire de Math\'ematiques d'Orsay\\ Universit\'e Paris-Sud\\ CNRS\\ Universit\'e Paris-Saclay\\ 91405 Orsay\\ FRANCE}
\email{cyril.houdayer@math.u-psud.fr}
\thanks{CH is supported by ERC Starting Grant GAN 637601}


\maketitle

%%%%%%%%%%
\section{Introduction}

{\em Free Araki--Woods factors} were introduced by Shlyakhtenko in \cite{Sh96}. In the framework of Voiculescu's free probability theory, these factors can be regarded as the type ${\rm III}$ counterparts of free group factors using Voiculescu's free Gaussian functor \cite{Vo85, VDN92}. Following \cite{Sh96}, to any orthogonal representation $U : \R \curvearrowright H_\R$ on a real Hilbert space, one  associates the {\em free Araki--Woods} von Neumann algebra $\Gamma(H_\R, U)\dpr$. The von Neumann algebra $\Gamma(H_\R, U)\dpr$ comes equipped with a unique {\it free quasi-free state} $\varphi_U$ which is always normal and faithful (see Section \ref{preliminaries} for a detailed construction). We have $\Gamma(H_\R,U)\dpr \cong\rL(\mathbf F_{\dim (H_\R)})$ when $U = 1_{H_\R}$ and $\Gamma(H_\R, U)\dpr$ is a full type ${\rm III}$ factor when $U \neq 1_{H_\R}$.

Let $U : \R \curvearrowright H_\R$ be any orthogonal representation. Using Zorn's lemma, we may decompose $H_\R = H_\R^{\ap} \oplus H_\R^{\wm}$ and $U = U^{\wm} \oplus U^{\ap}$ where $U^{\text{ap}} : \R \curvearrowright H_\R^{\ap}$ (resp.\ $U^{\wm} : \R \curvearrowright H_\R^{\wm}$) is the {\em almost periodic} (resp.\ {\em weakly mixing}) subrepresentation of $U : \R \curvearrowright H_\R$. Write $M = \Gamma(H_\R, U)\dpr$, $N = \Gamma(H_\R^{\ap}, U^{\ap})\dpr$ and $P = \Gamma(H_\R^{\wm}, U^{\wm})\dpr$ so that we have the following {\em free product} splitting
$$(M, \varphi_U) = (N, \varphi_{U^{\ap}}) \ast (P, \varphi_{U^{\wm}}).$$
Our main result provides a general structural decomposition for any von Neumann subalgebra $Q \subset M$ that is globally invariant under the modular automorphism group $\sigma^{\varphi_U}$ and shows that when $Q$ is moreover assumed to be amenable then $Q$ sits inside $N$. Our main theorem generalizes \cite[Theorem C]{HR14} to {\em arbitrary} free Araki--Woods factors.

\begin{mainthm}
Keep the same notation as above. Let $Q \subset M$ be any unital von Neumann subalgebra that is globally invariant under the modular automorphism group $\sigma^{\varphi_U}$. Then there exists a unique central projection $z \in \mathcal Z(Q) \subset M^{\varphi_U} = N^{\varphi_{U^{\ap}}}$ such that 
\begin{itemize}
\item $Qz$ is amenable and $Qz \subset zNz$ and

\item $Qz^\perp$ has no nonzero amenable direct summand and $(Q' \cap M^\omega)z^\perp = (Q' \cap M)z^\perp$ is atomic for any nonprincipal ultrafilter $\omega \in \beta(\N) \setminus \N$.
\end{itemize}
In particular, for any unital amenable von Neumann subalgebra $Q \subset M$ that is globally invariant under the modular automorphism group $\sigma^{\varphi_U}$, we have $Q \subset N$.
\end{mainthm}

Our main theorem should be compared to \cite[Theorem D]{Ho12b} which provides a similar result for crossed product ${\rm II_1}$ factors arising from free Bogoljubov actions of amenable groups. 

The core of our argument is Theorem \ref{thm-AOP-FAW} which generalizes \cite[Theorem 4.3]{HR14} to arbitrary free Araki--Woods factors. Let us point out that Theorem \ref{thm-AOP-FAW} is reminiscent of Popa's asymptotic orthogonality property in free group factors \cite{Po83} which is based on the study of central sequences in the ultraproduct framework. Unlike other results on this theme \cite{Ho12b,Ho14,HU15}, we do not assume here that the subalgebra $Q \subset M$ has a diffuse intersection with the free summand $N$ of the free product splitting $(M, \varphi_U) = (N, \varphi_{U^{\ap}}) \ast (P, \varphi_{U^{\wm}})$ and so we cannot exploit commutation relations of $Q$-central sequences with elements in $N$. Instead, we use the facts that $Q$ admits central sequences that are invariant under the modular automorphism group $\sigma^{\varphi_U^\omega}$ of the ultraproduct state $\varphi_U^\omega$ and that the modular automorphism group $\sigma^{\varphi_U}$ is weakly mixing on $P$.

\subsection*{Acknowledgments}
The present work was done when the authors were visiting the University of California at San Diego (UCSD). They thank Adrian Ioana and the Mathematics Department at UCSD for their kind hospitality.

%%%%%%%%%%
\section{Preliminaries}\label{preliminaries}

For any von Neumann algebra $M$, we denote by $\mathcal Z(M)$ the centre of $M$, by $\mathcal U(M)$ the group of unitaries in $M$, by $\Ball(M)$ the unit ball of $M$ with respect to the uniform norm and by $(M, \rL^2(M), J, \rL^2(M)_+)$ the standard form of $M$. We say that an inclusion of von Neumann algebras $P \subset M $ is {\em with expectation} if there exists a faithful normal conditional expectation $\rE_P :  M  \to P$. All the von Neumann algebras we consider in this paper are always assumed to $\sigma$-finite.

Let $M$ be any $\sigma$-finite von Neumann algebra with predual $M_\ast$ and $\varphi \in M_\ast$ any faithful state. We write $\|x\|_\varphi = \varphi(x^* x)^{1/2}$ for all $x \in M$. Recall that on $\Ball(M)$, the topology given by $\|\cdot\|_\varphi$ coincides with the $\sigma$-strong topology. Denote by $\xi_\varphi \in \rL^2(M)_+$ the unique representing vector of $\varphi$. The mapping $M \to \rL^2(M) : x \mapsto x \xi_\varphi$ defines an embedding with dense image such that $\|x\|_\varphi = \|x \xi_\varphi\|_{\rL^2(M)}$ for all $x \in M$. We denote by $\sigma^\varphi$ the modular automorphism group of the state  $\varphi$.  The {\em centralizer} $M^\varphi$ of the state $\varphi$ is by definition the fixed point algebra of $(M, \sigma^\varphi)$.  

Recall from \cite[Section 2.1]{Ho12a} that two subspaces $E, F \subset H$ of a Hilbert space are said to be $\varepsilon$-{\em orthogonal} for some $0 \leq \varepsilon \leq 1$ if $|\langle \xi, \eta\rangle| \leq \varepsilon \|\xi\| \|\eta\|$ for all $\xi \in E$ and all $\eta \in F$. We will then simply write $E \perp_\varepsilon F$.

%%%%%
\subsection*{Ultraproduct von Neumann algebras}
Let $M$ be any $\sigma$-finite von Neumann algebra and $\omega \in \beta(\N) \setminus \N$ any nonprincipal ultrafilter. Define
\begin{align*}
\mathcal I_\omega(M) &= \left\{ (x_n)_n \in \ell^\infty(M) : x_n \to 0\ \ast\text{-strongly as } n \to \omega \right\} \\
\mathcal M^\omega(M) &= \left \{ (x_n)_n \in \ell^\infty(M) :  (x_n)_n \, \mathcal I_\omega(M) \subset \mathcal I_\omega(M) \text{ and } \mathcal I_\omega(M) \, (x_n)_n \subset \mathcal I_\omega(M)\right\}.
\end{align*}
The {\em multiplier algebra} $\mathcal M^\omega(M)$ is a C$^*$-algebra and $\mathcal I_\omega(M) \subset \mathcal M^\omega(M)$ is a norm closed two-sided ideal. Following \cite[\S 5.1]{Oc85}, we define the {\em ultraproduct von Neumann algebra} $M^\omega$ by $M^\omega := \mathcal M^\omega(M) / \mathcal I_\omega(M)$, which is indeed known to be a von Neumann algebra. We denote the image of $(x_n)_n \in \mathcal M^\omega(M)$ by $(x_n)^\omega \in M^\omega$. 

For every $x \in M$, the constant sequence $(x)_n$ lies in the multiplier algebra $\mathcal M^\omega(M)$. We will then identify $M$ with $(M + \mathcal I_\omega(M))/ \mathcal I_\omega(M)$ and regard $M \subset M^\omega$ as a von Neumann subalgebra. The map $\rE_\omega : M^\omega \to M : (x_n)^\omega \mapsto \sigma \text{-weak} \lim_{n \to \omega} x_n$ is a faithful normal conditional expectation. For every faithful state $\varphi \in M_\ast$, the formula $\varphi^\omega := \varphi \circ \rE_\omega$ defines a faithful normal state on $M^\omega$. Observe that $\varphi^\omega((x_n)^\omega) = \lim_{n \to \omega} \varphi(x_n)$ for all $(x_n)^\omega \in M^\omega$.

Let $Q \subset M$ be any von Neumann subalgebra with faithful normal conditional expectation $\rE_Q : M \to Q$. Choose a faithful state $\varphi \in M_\ast$ in such a way that $\varphi = \varphi \circ \rE_Q$. We have $\ell^\infty(Q) \subset \ell^\infty(M)$, $\mathcal I_\omega(Q) \subset \mathcal I_\omega(M)$ and $\mathcal M^\omega(Q) \subset \mathcal M^\omega(M)$. We will then identify $Q^\omega = \mathcal M^\omega(Q) / \mathcal I_\omega(Q)$ with $(\mathcal M^\omega(Q) + \mathcal I_\omega(M)) / \mathcal I_\omega(M)$ and be able to regard $Q^\omega \subset M^\omega$ as a von Neumann subalgebra. Observe that the norm $\|\cdot\|_{(\varphi |_Q)^\omega}$ on $Q^\omega$ is the restriction of the norm $\|\cdot\|_{\varphi^\omega}$ to $Q^\omega$. Observe moreover that $(\rE_Q(x_n))_n \in \mathcal I_\omega(Q)$ for all $(x_n)_n \in \mathcal I_\omega(M)$ and $(\rE_Q(x_n))_n \in \mathcal M^\omega(Q)$ for all $(x_n)_n \in \mathcal M^\omega(M)$. Therefore, the mapping $\rE_{Q^\omega} : M^\omega \to Q^\omega : (x_n)^\omega \mapsto (\rE_Q(x_n))^\omega$ is a well-defined conditional expectation satisfying $\varphi^\omega \circ \rE_{Q^\omega} = \varphi^\omega$. Hence, $\rE_{Q^\omega} : M^\omega \to Q^\omega$ is a faithful normal conditional expectation. For more on ultraproduct von Neumann algebras, we refer the reader to \cite{AH12, Oc85}.

%%%%%
\subsection*{Free Araki--Woods factors}
\label{sectionFAW}

Let $H_{\R}$ be any real Hilbert space and $U : \R \curvearrowright H_\R$ any orthogonal representation. Denote by $H = H_{\R} \otimes_{\R} \C = H_\R \oplus {\rm i} H_\R$ the complexified Hilbert space, by $I : H \to H : \xi + {\rm i} \eta \mapsto \xi - {\rm i} \eta$ the canonical anti-unitary involution on $H$ and by $A$ the infinitesimal generator of $U : \R \curvearrowright H$, that is, $U_t = A^{{\rm i}t}$ for all $t \in \R$. Moreover, we have $IAI = A^{-1}$. Observe that $j : H_{\R} \to H :\zeta \mapsto (\frac{2}{A^{-1} + 1})^{1/2}\zeta$ defines an isometric embedding of $H_{\R}$ into $H$. Put $K_{\R} := j(H_{\R})$. It is easy to see that $K_\R \cap {\rm i} K_\R = \{0\}$ and that $K_\R + {\rm i} K_\R$ is dense in $H$. Write $T = I A^{-1/2}$. Then $T$ is a conjugate-linear closed invertible operator on $H$ satisfying $T = T^{-1}$ and $T^*T = A^{-1}$. Such an operator is called an {\it involution} on $H$. Moreover, we have $\dom(T) = \dom(A^{-1/2})$ and $K_\R = \{ \xi \in \dom(T) : T \xi = \xi \}$. In what follows, we will simply write
$$\overline{\xi + {\rm i} \eta} := T(\xi + {\rm i} \eta) = \xi - {\rm i} \eta, \forall \xi, \eta \in K_\R.$$

We introduce the \emph{full Fock space} of $H$:
\begin{equation*}
\mathcal{F}(H) =\C\Omega \oplus \bigoplus_{n = 1}^{\infty} H^{\otimes n}.
\end{equation*}
The unit vector $\Omega$ is called the \emph{vacuum vector}. For all $\xi \in H$, define the {\it left creation} operator $\ell(\xi) : \mathcal{F}(H) \to \mathcal{F}(H)$ by
\begin{equation*}
\left\{ 
{\begin{array}{l} \ell(\xi)\Omega = \xi, \\ 
\ell(\xi)(\xi_1 \otimes \cdots \otimes \xi_n) = \xi \otimes \xi_1 \otimes \cdots \otimes \xi_n.
\end{array}} \right.
\end{equation*}
We have $\|\ell(\xi)\|_\infty = \|\xi\|$ and $\ell(\xi)$ is an isometry if $\|\xi\| = 1$. For all $\xi \in K_\R$, put $W(\xi) := \ell(\xi) + \ell(\xi)^*$. The crucial result of Voiculescu \cite[Lemma 2.6.3]{VDN92} is that the distribution of the self-adjoint operator $W(\xi)$ with respect to the vector state $\varphi_U = \langle \, \cdot \,\Omega, \Omega\rangle$ is the semicircular law of Wigner supported on the interval $[-\|\xi\|, \|\xi\|]$. 

\begin{df}[Shlyakhtenko, \cite{Sh96}]
Let $H_\R$ be any real Hilbert space and $U : \R \curvearrowright H_\R$ any orthogonal representation. The \emph{free Araki--Woods} von Neumann algebra associated with $U : \R \curvearrowright H_\R$ is defined by
\begin{equation*}
\Gamma(H_\R, U)\dpr := \left\{W(\xi) : \xi \in K_{\R} \right\}\dpr.
\end{equation*}
We will denote by $\Gamma(H_{\R}, U)$ the unital $\rC^*$-algebra generated by $1$ and by all the elements $W(\xi)$ for $\xi \in K_\R$.
\end{df}

The vector state $\varphi_U = \langle \, \cdot \,\Omega, \Omega\rangle$ is called the {\it free quasi-free state} and is faithful on $\Gamma(H_\R, U)\dpr$. Let $\xi, \eta \in K_\R$ and write $\zeta = \xi + {\rm i} \eta$. Put
\begin{equation*}
W(\zeta) := W(\xi) +  {\rm i} W(\eta) = \ell(\zeta) + \ell(\overline \zeta)^*.
\end{equation*}
Note that the modular automorphism  group $\sigma^{\varphi_U}$ of the free quasi-free state $\varphi_U$ is given by $\sigma^{\varphi_U}_{t} = \Ad(\mathcal{F}(U_t))$, where $\mathcal{F}(U_t) = 1_{\C \Omega} \oplus \bigoplus_{n \geq 1} U_t^{\otimes n}$. In particular, it satisfies
\begin{equation*}
\sigma_{t}^{\varphi_U}(W(\zeta))  =  W(U_t \zeta), \forall \zeta \in K_\R +{\rm i} K_\R, \forall t \in \R. 
\end{equation*}

It is easy to see that for all $n \geq 1$ and all $\zeta_1, \dots, \zeta_n \in K_\R + {\rm i} K_\R$, $\zeta_1 \otimes \cdots \otimes \zeta_n \in \Gamma(H_\R, U)\dpr \Omega$. When $\zeta_1, \dots, \zeta_n$ are all nonzero, we will denote by $W(\zeta_1 \otimes \cdots \otimes  \zeta_n) \in \Gamma(H_\R, U)\dpr$ the unique element such that 
\[\zeta_1 \otimes \cdots \otimes \zeta_n = W(\zeta_1 \otimes \cdots \otimes \zeta_n) \Omega.\]
Such an element is called a {\it reduced word}. By \cite[Proposition 2.1 (i)]{HR14} (see also \cite[Proposition 2.4]{Ho12a}), the reduced word $W(\zeta_1 \otimes \cdots \otimes \zeta_n)$ satisfies the {\em Wick formula} given by
$$W(\zeta_1 \otimes \cdots \otimes \zeta_n) = \sum_{k = 0}^n \ell(\zeta_1) \cdots \ell(\zeta_k) \ell(\overline \zeta_{k + 1})^* \cdots \ell(\overline \zeta_n)^*.$$

Note that since inner products are assumed to be linear in the first variable, we have $\ell(\xi)^*\ell(\eta) = \overline{\langle \xi, \eta\rangle} 1 = \langle \eta, \xi \rangle 1$ for all $\xi, \eta \in H$. In particular, the Wick formula from \cite[Proposition 2.1 (ii)]{HR14} is 
\begin{align*}
& W(\xi_1 \otimes \cdots \otimes \xi_r) W(\eta_1 \otimes \cdots \otimes \eta_s) \\
&= W(\xi_1 \otimes \cdots \otimes \xi_r \otimes \eta_1 \otimes \cdots \otimes \eta_s) + 
      \overline{\langle \overline \xi_r, \eta_1\rangle} \, W(\xi_1 \otimes \cdots \otimes \xi_{r - 1}) W(\eta_2 \otimes \cdots \otimes \eta_s)
\end{align*}
for all $\xi_1, \dots, \xi_r, \eta_1, \dots, \eta_s \in K_\R + {\rm i} K_\R$. We will repeatedly use this fact  in the next section. We refer to \cite[Section 2]{HR14} for further details.

%%%%%%%%%%
\section{Asymptotic orthogonality property in free Araki--Woods factors}

Let $U : \R \curvearrowright H_\R$ be any orthogonal representation. Using Zorn's lemma, we may decompose $H_\R = H_\R^{\ap} \oplus H_\R^{\wm}$ and $U = U^{\wm} \oplus U^{\ap}$ where $U^{\text{ap}} : \R \curvearrowright H_\R^{\ap}$ (resp.\ $U^{\wm} : \R \curvearrowright H_\R^{\wm}$) is the {\em almost periodic} (resp.\ {\em weakly mixing}) subrepresentation of $U : \R \curvearrowright H_\R$. Write $M = \Gamma(H_\R, U)\dpr$, $N = \Gamma(H_\R^{\ap}, U^{\ap})\dpr$ and $P = \Gamma(H_\R^{\wm}, U_t^{\wm})\dpr$ so that 
$$(M, \varphi_U) = (N, \varphi_{U^{\ap}}) \ast (P, \varphi_{U^{\wm}}).$$
For notational convenience, we simply write $\varphi := \varphi_U$.

The main result of this section, Theorem \ref{thm-AOP-FAW} below, strengthens and generalizes \cite[Theorem 4.3]{HR14}.

\begin{thm}\label{thm-AOP-FAW}
Keep the same notation as above. Let $\omega \in \beta(\N) \setminus \N$ be any nonprincipal ultrafilter. For all $a \in M \ominus N$, all $b \in M$ and all $x, y \in (M^{\omega})^{\varphi^\omega} \cap (M^\omega \ominus M)$, we have
\[\varphi^\omega(b^*y^* a x) = 0.\]
\end{thm}

\begin{proof}
Denote as usual by $H := H_\R \otimes_\R \C$ the complexified Hilbert space and by $U : \R \curvearrowright H$ the corresponding unitary representation. Put $H^{\ap} := H_\R^{\ap} \otimes_\R \C$ and $H^{\wm} := H_\R^{\wm} \otimes_\R \C$. Put $K_\R := j(H_\R)$, $K_\R^{\ap} = j(H_\R^{\ap})$ and $K_\R^{\wm} := j(H_\R^{\wm})$, where $j$ is the isometric embedding $\xi \in H_\R \mapsto (\frac{2}{1 + A^{-1}})^{1/2} \xi \in H$. Denote by $\mathcal H = \mathcal F(H)$ the full Fock space of $H$. For every $t \in \R$, put $\kappa_t = 1_{\C \Omega} \oplus \bigoplus_{n \geq 1} U_t^{\otimes n} \in \mathcal U(\mathcal H)$. For every $t \in \R$ and every $x \in M$, we have $\sigma_t^\varphi(x) \Omega = \kappa_t(x \Omega)$. We will implicitly identify the full Fock space $\mathcal F(H)$ with the standard Hilbert space $\rL^2(M)$ and the vacuum vector $\Omega \in \mathcal H$ with the canonical representing vector $\xi_\varphi \in \rL^2(M)_+$.

Put $K_{\an} := \bigcup_{\lambda > 1} \mathbf 1_{[\lambda^{-1}, \lambda]}(A)(K_\R + {\rm i} K_\R)$. Observe that $K_{\an} \subset K_\R + {\rm i} K_\R$ is a dense subspace of elements $\eta \in K_\R + {\rm i} K_\R$ for which the map $\R \to K_\R + {\rm i} K_\R : t \mapsto U_t \eta$ extends to an $(K_\R + {\rm i} K_\R)$-valued entire analytic function and that $\overline{K_{\an}} = K_{\an}$. For all $\eta \in K_{\an}$, the element $W(\eta)$ is analytic with respect to the modular automorphism group $\sigma^{\varphi}$ and we have $\sigma_z^{\varphi}(W(\eta)) = W(A^{{\rm i}z}\eta)$ for all $z \in \C$.

Denote by $\mathcal W$ the set of reduced words of the form $W(\xi_1 \otimes \cdots \otimes \xi_n)$ for which $n \geq 1$, $\xi_1, \dots, \xi_n \in K_{\an}$. By linearity/density, in order to prove Theorem \ref{thm-AOP-FAW}, we may assume without loss of generality that $a$ and $b$ are reduced words in $\mathcal W$. Since moreover $a \in M \ominus N$, we can assume that at least one of its letters $\xi_i$ lies in $K_\R^{\wm} + {\rm i} K_\R^{\wm}$. More precisely, we can write
\begin{align*}
a &= a' \, W(\xi_{1} \otimes \cdots \otimes \xi_{p}) \, a\dpr \\
b &= b' \, W(\eta_{1} \otimes \cdots \otimes \eta_{q}) \, b\dpr
\end{align*}
with $p \geq 1$, $q \geq 0$ $a', a\dpr, b', b\dpr$ are reduced words in $N$ with letters in $K_{\an} \cap (K_\R^{\ap} + {\rm i} K_\R^{\ap})$, $\xi_{2}, \dots, \xi_{p - 1}$, $\eta_{2}, \dots, \eta_{q - 1} \in K_{\an}$ and $\xi_{1}, \xi_{p}, \eta_{1}, \eta_{q} \in K_{\an} \cap (K_\R^{\wm} + {\rm i} K_\R^{\wm})$. By convention, when $q = 0$, $W(\eta_{1} \otimes \cdots \otimes \eta_{q})$ is the trivial word $1$, so that $b = b'b\dpr$.

Denote by $L \subset K_\R^{\wm} + {\rm i} K_\R^{\wm}$ the finite dimensional subspace generated by $\xi_{1}, \xi_{p}, \eta_{1}, \eta_{q}$ and such that $\overline L = L$. If $q = 0$, then $L$ is simply the subspace generated by $\xi_{1}, \xi_{p}, \overline \xi_1, \overline \xi_p$. 

Denote by 
\begin{itemize}
\item $\mathcal X(1, r) \subset \mathcal H$ the closed linear subspace generated by all the reduced words of the form $e_1 \otimes \cdots \otimes e_n$ with $r \geq 0$, $n \geq r + 1$, $e_1, \dots, e_r \in K_\R^{\ap} + {\rm i}K_\R^{\ap}$ and $e_{r + 1} \in L$. When $r = 0$, simply denote $\mathcal X_1 := \mathcal X(1, 0)$.
\item $\mathcal X(2, r) \subset \mathcal H$ the closed linear subspace generated by all the reduced words of the form $e_1 \otimes \cdots \otimes e_n$ with $r \geq 0$, $n \geq r + 1$, $e_{n - r} \in L$ and $e_{n - r + 1}, \dots, e_n \in K_\R^{\ap} + {\rm i}K_\R^{\ap}$. When $r = 0$, simply denote $\mathcal X_2 := \mathcal X(2, 0)$.
\item $\mathcal Y \subset \mathcal H$ the closed linear subspace generated by all the reduced words of the form $e_1 \otimes \cdots \otimes e_n$ with $n \geq 1$ and $e_1, e_n  \in L^\perp$.
\end{itemize}
Observe that we have the following orthogonal decomposition
$$\mathcal H = \C \Omega \oplus \overline{\left (\mathcal X_1 + \mathcal X_2 \right)} \oplus \mathcal Y.$$

\begin{claim}\label{claim1}
Let $\varepsilon \geq 0$ and $t \in \R$ such that $U_t (L) \perp_{\varepsilon/\dim L} L$. Then for all $i \in \{1, 2\}$ and all $r \geq 0$, we have
$$\kappa_t (\mathcal X(i, r)) \perp_\varepsilon \mathcal X(i, r).$$ 
\end{claim}

\begin{proof}[Proof of Claim \ref{claim1}]
Choose an orthonormal basis $(\zeta_1, \dots, \zeta_{\dim L})$ of $L$. We first prove the claim for $i = 1$. We will identify $\mathcal X(1, r)$ with $L \otimes ((H^{\ap})^{\otimes r} \otimes \mathcal H)$ using the following unitary defined by
$$\mathcal V(1, r) : H \otimes (H^{\otimes r} \otimes \mathcal H) \to \mathcal H : \zeta \otimes \mu \otimes \nu \mapsto \mu \otimes \zeta \otimes \nu.$$
Observe that $\kappa_t \mathcal V(1, r) = \mathcal V(1, r) (U_t \otimes (U_t)^{\otimes r} \otimes \kappa_t)$ for every $t \in \R$. Let $\Xi_1, \Xi_2 \in \mathcal X(1, r)$ be such that $\Xi_1 = \sum_{i = 1}^{\dim L}  \zeta_i \otimes \Theta_i^1$ and $\Xi_2 = \sum_{j = 1}^{\dim L} \zeta_j \otimes \Theta_j^2$ with $\Theta_i^1, \Theta_j^2 \in (H^{\ap})^{\otimes r} \otimes \mathcal H$. We have $\kappa_t(\Xi_1) = \sum_{i = 1}^{\dim L} U_t(\zeta_i) \otimes \kappa_t(\Theta_i^1)$ and hence
$$|\langle \kappa_t(\Xi_1), \Xi_2\rangle| \leq \sum_{i,j = 1}^{\dim L} |\langle U_t(\zeta_i), \zeta_j \rangle| \|\Theta_i^1\| \|\Theta_j^2\|.$$
Since $|\langle U_t(\zeta_i), \zeta_j\rangle| \leq \varepsilon / \dim L$, we obtain $|\langle \kappa_t(\Xi_1), \Xi_2\rangle| \leq \varepsilon \|\Xi_1\| \|\Xi_2\|$ by the Cauchy--Schwarz inequality. The proof of the claim for $i = 2$ is entirely analogous.
\end{proof}

Given a closed subspace $\mathcal K \subset \cH$, we denote by $P_{\mathcal K} : \cH \to \mathcal K$ the orthogonal projection onto $\mathcal K$.

\begin{claim}\label{claim2}
Take $z = (z_n)^\omega \in (M^\omega)^{\varphi^\omega}$ and let $w_1, w_2 \in N$ be any elements of the following form:
\begin{itemize}
\item either $w_1 = 1$ or $w_1 = W(\zeta_1 \otimes \cdots \otimes \zeta_r)$ with $r \geq 1$ and $\zeta_1, \dots, \zeta_r \in K_{\an} \cap (K_\R^{\ap} +{\rm i} K_\R^{\ap})$.
\item either $w_2 = 1$ or $w_2 = W(\mu_1 \otimes \cdots \otimes \mu_s)$ with $s \geq 1$ and $\mu_1, \dots, \mu_s \in K_{\an} \cap (K_\R^{\ap} +{\rm i} K_\R^{\ap})$.
\end{itemize}
Then for all $i \in \{1, 2\}$, we have $\lim_{n \to \omega} \|P_{\mathcal X_i}(w_1z_nw_2 \Omega)\| = 0$.
\end{claim}

\begin{proof}[Proof of Claim \ref{claim2}]
Observe that $w_1 z_n w_2 \Omega = w_1 J \sigma_{-{\rm i}/2}^\varphi(w_2^*) J \, z_n \Omega$. Firstly, we have
\begin{align*}
P_{\mathcal X(1, r)}(J \sigma_{-{\rm i}/2}^\varphi(w_2^*) J \, z_n \Omega) &= J \sigma_{-{\rm i}/2}^\varphi(w_2^*) J P_{\mathcal X(1, r)}(z_n \Omega) \\
P_{\mathcal X(2, s)}(w_1 \, z_n \Omega) &= w_1 P_{\mathcal X(2, s)}(z_n \Omega).
\end{align*}
Secondly, for all $\Xi \in \mathcal H$, we have
\begin{align*}
P_{\cX_1}(w_1\Xi) & = P_{\mathcal X_1}(w_1 \, P_{\mathcal X(1, r)}(\Xi)) \\
P_{\cX_2}(J \sigma_{-{\rm i}/2}^\varphi(w_2^*) J\Xi) & = P_{\mathcal X_2}(J \sigma_{-{\rm i}/2}^\varphi(w_2^*) J \,P_{\mathcal X(2, s)}(\Xi)).
\end{align*}
This implies that
\begin{align*}
P_{\mathcal X_1}(w_1 z_n w_2 \Omega) &= P_{\mathcal X_1}(w_1 J \sigma_{-{\rm i}/2}^\varphi(w_2^*) J \, P_{\mathcal X(1, r)}(z_n \Omega)) \\
P_{\mathcal X_2}(w_1 z_n w_2 \Omega) &= P_{\mathcal X_2}(w_1 J \sigma_{-{\rm i}/2}^\varphi(w_2^*) J \, P_{\mathcal X(2, s)}(z_n \Omega)),
\end{align*}
and we are left to show that $\lim_{n \to \omega} \|P_{\mathcal X(1, r)}(z_n \Omega)\| = \lim_{n \to \omega} \|P_{\mathcal X(2, s)}(z_n \Omega)\| = 0$.

Let $i \in \{1, 2\}$ and $k \in \{r, s\}$. Fix $N \geq 0$. Since the orthogonal representation $U : \R \curvearrowright H_\R^{\wm}$ is weakly mixing and $L \subset H^{\wm}$ is a finite dimensional subspace, we may choose inductively $t_1, \dots, t_{N} \in \R$ such that $U_{t_{j_1}}(L) \perp_{(N\dim(L))^{-1}} U_{t_{j_2}}(L)$ for all $1 \leq j_1 < j_2 \leq N$. By Claim \ref{claim1}, this implies that \[\kappa_{t_{j_1}}(\mathcal X(i, k)) \perp_{1/N} \kappa_{t_{j_2}}(\mathcal X(i, k)), \forall 1 \leq j_1 < j_2 \leq N.\]
For all $t \in \R$ and all $n \in \N$, we have
\begin{align*}
\|P_{\mathcal X(i, k)}(z_n \Omega)\|^2 &= \langle P_{\mathcal X(i, k)}(z_n \Omega), z_n \Omega\rangle \\
&= \langle \kappa_t(P_{\mathcal X(i, k)}(z_n \Omega)), \kappa_t(z_n\Omega) \rangle \quad (\text{since } \kappa_t \in \mathcal U(\mathcal H))\\
&= \langle P_{\kappa_t(\mathcal X(i, k))}(\kappa_t(z_n \Omega)), \kappa_t(z_n \Omega) \rangle.
\end{align*}
By \cite[Theorem 4.1]{AH12}, for all $t \in \R$, we have $(z_n)^\omega = z = \sigma_t^{\varphi^\omega}(z) = (\sigma_t^\varphi(z_n))^\omega$. This implies that $\lim_{n \to \omega} \|\sigma_t^\varphi(z_n) - z_n\|_\varphi = 0$, and hence $\lim_{n \to \omega} \|\kappa_t(z_n\Omega) - z_n\Omega\| = 0$ for all $t \in \R$. In particular, since the sequence $(z_n \Omega)_n$ is bounded in $\mathcal H$, we deduce that for all $t \in \R$,
\[\lim_{n \to \omega} \|P_{\mathcal X(i, k)}(z_n \Omega)\|^2 = \lim_{n \to \omega} \langle P_{\kappa_t(\mathcal X(i, k))}(z_n \Omega), z_n \Omega \rangle.\]

Applying this equality to our well chosen reals $(t_j)_{1 \leq j \leq N}$, taking a convex combination and applying Cauchy--Schwarz inequality, we obtain
\begin{align*}
\lim_{n \to \omega} \|P_{\mathcal X(i, k)}(z_n \Omega)\|^2 &= \lim_{n \to \omega} \frac{1}{N}\sum_{j = 1}^{N} \langle P_{\kappa_{t_j}(\mathcal X(i, k))}(z_n \Omega), z_n\Omega\rangle \\
& =  \lim_{n \to \omega} \frac{1}{N} \left\langle \sum_{j = 1}^{N} P_{\kappa_{t_j}(\mathcal X(i, k))}(z_n \Omega), z_n\Omega \right\rangle \\
& \leq \lim_{n \to \omega} \frac{1}{N}\left\|\sum_{j = 1}^{N} P_{\kappa_{t_j}(\mathcal X(i, k))}(z_n \Omega)\right\| \|z_n\|_\varphi.
\end{align*}
Then, for all $n \in \N$ we have,
\begin{align*}
\left\|\sum_{j = 1}^{N} P_{\kappa_{t_j}(\mathcal X(i, k))}(z_n \Omega)\right\|^2 & = \sum_{j_1,j_2 = 1}^N \langle P_{\kappa_{t_{j_1}}(\mathcal X(i, k))}(z_n \Omega),P_{\kappa_{t_{j_2}}(\mathcal X(i, k))}(z_n \Omega)\rangle\\
& \leq \sum_{j=1}^N \|P_{\kappa_{t_j}(\mathcal X(i, k))}(z_n \Omega)\|^2 + \sum_{j_1 \neq j_2}^N \frac{\Vert z_n \Vert_\varphi^2}{N}\\
&\leq N\Vert z_n \Vert_\varphi^2 + N^2\frac{\Vert z_n \Vert_\varphi^2}{N} \\
&= 2N\Vert z_n \Vert_\varphi^2.
\end{align*}
Altogether, we have obtained the inequality $\lim_{n \to \omega} \|P_{\mathcal X(i, k)}(z_n \Omega)\|^2 \leq \sqrt{2}\Vert z  \Vert_{\varphi^\omega}^2/\sqrt{N}$. As $N$ is arbitrarily large, this finishes the proof of Claim \ref{claim2}. The above argument is inspired from \cite[Lemma 10]{We15}. Alternatively, we could have used \cite[Proposition 2.3]{Ho12a}.
\end{proof}

\begin{claim}\label{claim3}
The subspaces $W(\xi_{1} \otimes \cdots \otimes \xi_{p}) \mathcal Y$ and $J \sigma_{-{\rm i}/2}^\varphi (W(\overline \eta_q \otimes \cdots \otimes \overline \eta_1)) J \mathcal Y$ are orthogonal in $\mathcal H$. Here, in the case $q = 0$, the vector space $J \sigma_{-{\rm i}/2}^\varphi (W(\overline \eta_q \otimes \cdots \otimes \overline \eta_1)) J \mathcal Y$ is nothing but $\cY$.
\end{claim}

\begin{proof}[Proof of Claim \ref{claim3}]
Let $m, n \geq 1$, $e_1, \dots, e_m, f_1, \dots, f_n \in H$ with $e_1,e_m, f_1, f_n \in L^\perp$ so that the vectors $e_1 \otimes \dots \otimes e_m$ and $f_1 \otimes \dots \otimes f_n$ belong to $\mathcal Y$. Since $\overline \xi_p \perp e_1$, $\overline f_n \perp \eta_1$ and $\xi_1 \perp f_1$, we have
\begin{align*}
& \langle W(\xi_{1} \otimes \cdots \otimes \xi_{p}) \, (e_1 \otimes \dots \otimes e_m), J \sigma_{-{\rm i}/2}^\varphi (W(\overline \eta_q \otimes \cdots \otimes \overline \eta_1)) J \, (f_1 \otimes \dots \otimes f_n)\rangle \\
&=  \langle W(\xi_{1} \otimes \cdots \otimes \xi_{p}) W(e_1 \otimes \dots \otimes e_m) \Omega, J \sigma_{-{\rm i}/2}^\varphi (W(\overline \eta_q \otimes \cdots \otimes \overline \eta_1)) J W(f_1 \otimes \dots \otimes f_n)\Omega\rangle \\
&=  \langle W(\xi_{1} \otimes \cdots \otimes \xi_{p}) W(e_1 \otimes \dots \otimes e_m) \Omega,  W(f_1 \otimes \dots \otimes f_n)W(\eta_1 \otimes \cdots \otimes \eta_q)\Omega\rangle \\
&= \langle W(\xi_{1} \otimes \cdots \otimes \xi_{p} \otimes e_1 \otimes \dots \otimes e_m) \Omega, W(f_1 \otimes \dots \otimes f_n \otimes \eta_1 \otimes \cdots \otimes \eta_q) \Omega\rangle \\
&= \langle \xi_{1} \otimes \cdots \otimes \xi_{p} \otimes e_1 \otimes \dots \otimes e_m, f_1 \otimes \dots \otimes f_n \otimes \eta_1 \otimes \cdots \otimes \eta_q\rangle \\
&= 0.
\end{align*}
Note that in the case $q = 0$, the above calculation still makes sense. Indeed we have
\[\langle W(\xi_{1} \otimes \cdots \otimes \xi_{p}) \, (e_1 \otimes \dots \otimes e_m), (f_1 \otimes \dots \otimes f_n)\rangle = \langle \xi_{1} \otimes \cdots \otimes \xi_{p} \otimes e_1 \otimes \dots \otimes e_m, f_1 \otimes \dots \otimes f_n \rangle = 0.\]

Since the linear span of all such reduced words $e_1 \otimes \dots \otimes e_m$ (resp.\ $f_1 \otimes \dots \otimes f_n$) generate $\mathcal Y$, we obtain that the subspaces $W(\xi_{1} \otimes \cdots \otimes \xi_{p}) \mathcal Y$ and $J \sigma_{-{\rm i}/2}^\varphi (W(\overline \eta_q \otimes \cdots \otimes \overline \eta_1)) J \mathcal Y$ are orthogonal in $\mathcal H$.
\end{proof}

Let $x, y \in (M^\omega)^{\varphi^\omega}  \cap (M^\omega \ominus M)$. We have
\begin{align*}
 \varphi^\omega(b^* y^* a x) &= \langle ax \xi_{\varphi^\omega}, yb\xi_{\varphi^\omega}\rangle  \\
&= \lim_{n \to \omega} \langle ax_n \xi_\varphi, y_nb \xi_\varphi\rangle  \\
& =\lim_{n \to \omega} \langle a' W(\xi_1 \otimes \cdots \otimes \xi_p) a\dpr \, x_n \Omega, y_n \, b' W(\eta_1 \otimes \cdots \otimes \eta_q) b\dpr\Omega \rangle \\
& = \lim_{n \to \omega} \langle  W(\xi_1 \otimes \cdots \otimes \xi_p) \, a\dpr x_n \sigma_{-{\rm i}}^\varphi((b\dpr)^*) \Omega, J \sigma_{-{\rm i}/2}^\varphi(W(\overline \eta_q \otimes \cdots \otimes \overline \eta_1)) J \, (a')^* y_n b' \Omega \rangle.
\end{align*}
Put $z_n = a\dpr x_n \sigma_{-{\rm i}}^\varphi((b\dpr)^*)$ and $z'_n = (a')^* y_n b'$.
By Claim \ref{claim2}, we have that $\lim_{n \to \omega} \|P_{\mathcal X_i}(z_n \Omega)\| = \lim_{n \to \omega} \|P_{\cX_i}(z'_n \Omega)\| = 0$ for all $i \in \{ 1,2\}$. Since moreover $\rE_\omega(x) = \rE_\omega(y) = 0$, we see that $\lim_{n \to \omega} \|P_{\C\Omega}(z_n \Omega)\| = \lim_{n \to \omega} \|P_{\C\Omega}(z'_n \Omega)\| = 0$. Since $\mathcal H = \C \Omega \oplus \overline{\left (\mathcal X_1 + \mathcal X_2 \right)} \oplus \mathcal Y$, we obtain 
\[\lim_{n \to \omega} \Vert z_n \Omega - P_{\cY}(z_n \Omega)\Vert = 0 \qquad \text{and} \qquad \lim_{n \to \omega} \Vert z'_n \Omega - P_{\cY}(z'_n \Omega)\Vert = 0.\]
By Claim \ref{claim3}, we finally obtain
\begin{align*}
\varphi^\omega(b^* y^* a x) & = \lim_{n \to \omega} \langle  W(\xi_1 \otimes \cdots \otimes \xi_p) \, z_n \Omega, J \sigma_{-{\rm i}/2}^\varphi(W(\overline \eta_q \otimes \cdots \otimes \overline \eta_1)) J \,z'_n \Omega \rangle \\
&= \lim_{n \to \omega} \langle  W(\xi_1 \otimes \cdots \otimes \xi_p) \, P_{\cY}(z_n \Omega), J \sigma_{-{\rm i}/2}^\varphi(W(\overline \eta_q \otimes \cdots \otimes \overline \eta_1)) J \,P_{\cY}(z'_n \Omega) \rangle \\
&= 0.
\end{align*}
This finishes the proof of Theorem \ref{thm-AOP-FAW}.
\end{proof}

%%%%%%%%%%
\section{Proof of the main theorem}

We start by proving the following intermediate result.

\begin{thm}\label{thm-amenable}
Let $(M, \varphi) = (\Gamma(H_\R, U)\dpr, \varphi_U)$ be any free Araki--Woods factor endowed with its free quasi-free state. Keep the same notation as in the introduction. Let $q \in M^\varphi = N^{\varphi_{U^{\ap}}}$ be any nonzero projection. Write $\varphi_q = \frac{\varphi(q \, \cdot \, q)}{\varphi(q)}$.

Then for any amenable von Neumann subalgebra $Q \subset qMq$ that is globally invariant under the modular automorphism group $\sigma^{\varphi_q}$, we have $Q \subset qNq$.
\end{thm}
\begin{proof}
We may assume that $Q$ has separable predual. Indeed, let $x \in Q$ be any element and denote by $Q_0 \subset Q$ the von Neumann subalgebra generated by $x \in Q$ and that is globally invariant under the modular automorphism group $\sigma^{\varphi_q}$. Then $Q_0$ is amenable and has separable predual. Therefore, we may assume without loss of generality that $Q_0 = Q$, that is, $Q$ has separable predual.

{\bf Special case.} We first prove the result when $Q \subset qMq$ is globally invariant under $\sigma^{\varphi_q}$ and is an irreducible subfactor meaning that $Q' \cap qMq = \C q$. 

Let $a \in Q$ be any element. Since $Q$ is amenable and has separable predual, $Q' \cap (qMq)^\omega$ is diffuse and so is $Q' \cap ((qMq)^\omega)^{\varphi_q^\omega}$ by \cite[Theorem 2.3]{HR14}. In particular, there exists a unitary $u \in \mathcal U(Q' \cap ((qMq)^\omega)^{\varphi_q^\omega})$ such that $\varphi_q^\omega(u) = 0$. Note that $\rE_\omega(u) \in Q' \cap qMq = \C q$ and hence $\rE_\omega(u) = \varphi_q^\omega(u) = 0$ so that $u \in (M^\omega)^{\varphi^\omega} \cap (M^\omega \ominus M)$. Theorem \ref{thm-AOP-FAW} yields $\varphi^\omega(a^*u^*(a - \rE_N(a))u) = 0$. Since moreover $au = ua$ and $u \in \mathcal U((qMq)^{\varphi_q^\omega})$, we have
\begin{align*}
\|a\|_\varphi^2 &= \| au\|^2_{\varphi^\omega} \\
&= \varphi^\omega(u^* a^* a u) = \varphi^\omega(a^*u^*au) \\
&= \varphi^\omega(a^*u^*\rE_N(a)u) = \varphi^\omega(ua^*u^*\rE_N(a)) \\
&= \varphi(a^* \rE_N(a)) \\
&= \|\rE_N(a)\|_\varphi^2.
\end{align*}
This shows that $a = \rE_N(a) \in N$.

{\bf General case.} We next prove the result when $Q \subset qMq$ is any amenable subalgebra globally invariant under $\sigma^{\varphi_q}$.

Denote by $z \in \mathcal Z(Q) \subset N^\varphi$ the unique central projection such that $Qz$ is atomic and $Q(1-z)$ is diffuse. Since $Qz$ is atomic and globally invariant under the modular automorphism group $\sigma^{\varphi_z}$, we have that $\varphi_z|_{Qz}$ is almost periodic and hence $Qz \subset N$. It remains to prove that $Q(1-z) \subset N$. Cutting down by $1 - z$ if necessary, we may assume that $Q$ itself is diffuse.

Since $Q \subset qMq$ is diffuse and with expectation and since $M$ is solid (see \cite[Theorem A]{HR14} and \cite[Theorem 7.1]{HI15} which does not require separability of the predual), the relative commutant $Q' \cap qMq$ is amenable. Up to replacing $Q$ by $Q \vee Q' \cap qMq$ which is still amenable and globally invariant under the modular automorphism group $\sigma^{\varphi_q}$, we may assume that $Q' \cap qMq = \mathcal Z(Q)$. Denote by $(z_n)_n$ a sequence of central projections in $\mathcal Z(Q)$ such that $\sum_n z_n = q$, $(Qz_0)' \cap z_0Mz_0 = \mathcal Z(Q)z_0$ is diffuse and $(Qz_n)' \cap z_n M z_n = \C z_n$ for every $n \geq 1$. 
 \begin{itemize}
 \item By the {\bf Special case} above, we know that $Qz_n \subset N$ for all $n \geq 1$.
\item Since $\cZ(Q)z_0 \oplus (1-z_0)N(1-z_0)$ is diffuse and with expectation in $N$, its relative commutant inside $M$ is contained in $N$ by \cite[Proposition 2.7(1)]{HU15}. In particular, we have $Qz_0 \subset N$.
\end{itemize}
Therefore, we have $Q \subset N$.
\end{proof}

\begin{proof}[Proof of the main theorem]
Put $\varphi := \varphi_U$. Denote by $z \in \mathcal Z(Q) \subset M^\varphi = N^\varphi$ the unique central projection such that $Qz$ is amenable and $Qz^\perp$ has no nonzero amenable direct summand. By Theorem \ref{thm-amenable}, we have $Qz \subset zNz$. Next, fix $\omega \in \beta(\N) \setminus \N$ any nonprincipal ultrafilter. By \cite[Theorem A]{HR14} (see also \cite[Theorem 7.1]{HI15}), we have that $(Q' \cap M^\omega)z^\perp = (Q' \cap M)z^\perp$ is atomic.
\end{proof}

\bibliographystyle{plain}

\end{document}